\documentclass[11pt, a4paper, reqno, oneside]{article}

\usepackage{amssymb,amsmath}
\DeclareUnicodeCharacter{0335}{}
\usepackage{tikz-cd}
\usepackage{authblk} 
\usepackage{physics} 
\usepackage{tikz}
\usepackage{titlesec} 
\usetikzlibrary{babel}
\usepackage{adjustbox}
\usepackage{url}
\usepackage{xcolor}
\usepackage{yhmath}
\usepackage{eso-pic}
\usepackage[T1]{fontenc}
\usepackage[utf8]{inputenc}
\usepackage[spanish,english]{babel}
\usepackage{csquotes}
\usepackage{latexsym}
\usepackage{euscript}
\usepackage{mathtools}
\usepackage[all]{xy}
\usepackage{mathrsfs}
\usepackage[percent]{overpic}
\usepackage{setspace}
\usepackage{graphicx}
\usepackage{amsthm}
\usepackage{yfonts}
\usetikzlibrary{calc}
\usetikzlibrary{decorations.pathmorphing}

\tikzset{curve/.style={settings={#1},to path={(\tikztostart)
    .. controls ($(\tikztostart)!\pv{pos}!(\tikztotarget)!\pv{height}!270:(\tikztotarget)$)
    and ($(\tikztostart)!1-\pv{pos}!(\tikztotarget)!\pv{height}!270:(\tikztotarget)$)
    .. (\tikztotarget)\tikztonodes}},
    settings/.code={\tikzset{quiver/.cd,#1}
        \def\pv##1{\pgfkeysvalueof{/tikz/quiver/##1}}},
    quiver/.cd,pos/.initial=0.35,height/.initial=0}

\tikzset{tail reversed/.code={\pgfsetarrowsstart{tikzcd to}}}
\tikzset{2tail/.code={\pgfsetarrowsstart{Implies[reversed]}}}
\tikzset{2tail reversed/.code={\pgfsetarrowsstart{Implies}}}
\tikzset{no body/.style={/tikz/dash pattern=on 0 off 1mm}}

\usepackage{hyperref}
\usepackage{orcidlink}
\ExplSyntaxOn
\tl_new:N \l_main_bookmark_tl
\tl_const:Nx \c_main_subscript_tl { \char_generate:nn {95} {8} }
\cs_new_eq:NN \main_pdfstringdef:Nn \pdfstringdef
\cs_set_protected:Npn \pdfstringdef #1#2
  {
    \group_begin:
      \tl_set:Nn \l_main_bookmark_tl {#2}
      \tl_replace_all:Nnn \l_main_bookmark_tl {$} {}
      \tl_replace_all:Nnn \l_main_bookmark_tl {^} {\textasciicircum}
      \tl_replace_all:NVn \l_main_bookmark_tl \c_main_subscript_tl {\textunderscore}
      \exp_args:NNV \main_pdfstringdef:Nn #1 \l_main_bookmark_tl
    \group_end:
  }
\ExplSyntaxOff
\normalbaselines
\makeatletter

\def\ps@myfancy{\let\@mkboth\markboth
 \def\@evenhead{\vbox{\hsize\textwidth 
 \hbox to \textwidth{\sf\mdseries\thepage 
 \rule[-.6ex]{0mm}{2mm} \hfill\sf\large\leftmark}
 \vskip 1pt \hrule}}
 \def\@oddhead{\vbox{\hsize\textwidth 
 \hbox to \textwidth{{\sf\large\leftmark}
 \rule[-.6ex]{0mm}{2mm} \hfill\sf\mdseries{\thepage}}
 \vskip 1pt \hrule}}}

\def\ps@myfancyplain{
 \def\@evenhead{\vbox{\hsize\textwidth%
 \rule[-.6ex]{0mm}{2mm} \hfill }
 \vskip 1pt \hrule
 \vskip\headsep
 \vskip\textheight
 \vskip1pc
 \hbox to \textwidth{\sf\mdseries\thepage 
 \rule[-6ex]{0mm}{2mm} \hfill }}
 \def\@oddhead{\vbox{\hsize\textwidth 
 \vskip 1pt\hrule
 \vskip\headsep
 \vskip\textheight
 \vskip2pc
 \hbox to \textwidth{\hfill\rule[.4ex]{1pc}{2.5pt}
 \sf\mdseries\thepage}
}}}

\def\ps@myemptyfun{%
 \let\@oddhead\@empty
 \let\@evenhead\@empty
 \def\@oddfoot{\hfil\thepage\hfil}%
 \let\@evenfoot\@oddfoot}

\makeatother
\providecommand{\proofname}{Demostraci\'on.}
 {\par\noindent{\it Demostraci\'on. }\nopagebreak\normalsize}%
 {}

 {\par\noindent{\it #1. }\nopagebreak\normalsize}%
 {\hfill\linebreak[2]\hspace*{\fill}$\square$\\[-1pt]}
\makeatother

\def\sqbullet{\raise.2ex\hbox{\vrule width 3.5pt height 3.5pt}}

{}

{\em}

\newcounter{substep}
\def\thesubstep{\arabic{substep}}

{\em}

\newcounter{subsubstep}
\def\thesubsubstep{\arabic{subsubstep}}

{\em}

\numberwithin{figure}{section}

{}

\newtheoremstyle{mystyle}
  {}
  {}
  {\itshape}
  {}
  {\sf \bfseries}
  {}
{ }
  {\thmname{#1}\thmnumber{{\textcolor{blue}{\, \hspace{-1mm}#2.}}}\thmnote{ (#3)}}

\theoremstyle{mystyle}
\definecolor{royalblue(web)}{rgb}{0.25, 0.41, 0.88}
\hypersetup{
    colorlinks=true,
    linktocpage=true,
    citecolor=black,
    linkcolor=royalblue(web),
    filecolor=magenta,      
    urlcolor=cyan,
    pdftitle={Overleaf Example},
    pdfpagemode=FullScreen,
    }

\titleformat{\section}
  {\normalfont\LARGE\bfseries}{\thesection.}{1em}{}
\titleformat{\subsection}
  {\normalfont\Large\bfseries}{\thesubsection.}{1em}{}
\titleformat{\subsubsection}
  {\normalfont\normalsize\itshape}{\thesubsubsection.}{1em}{}

\newtheorem{Teor}{Theorem}[section]

\newtheorem{Coro}[Teor]{Corollary}

\newcommand{\R}{{\mathbb R}}
 \newcommand{\C}{{\mathbb C}}

\newcommand{\mail}[1]{\small\href{mailto:#1}{#1}}

\newenvironment{Abstract}
{
\begin{center}
\textbf{Abstract}\\
\vspace{0.25cm}
\begin{minipage}{14.5cm}}
{\footnotesize
\end{minipage}
\end{center}}
\begin{document}


	\begin{center}
		{\huge {\bfseries A short note about the spaces $BV^s(\R^n)$\par}}
		\vspace{1cm}
		
\begin{center} 
	 {\Large Guillermo García-Sáez}{\small\textsuperscript{1}} \\
	\mail{guillermo.garciasaez@uclm.es}\:\orcidlink{0009-0008-4335-119X}
\end{center}
\vspace{5mm}

\textsc{\textsuperscript{1}ETSII, Departamento de Matem\'aticas\\ Universidad de Castilla-La Mancha} \\
		Campus Universitario s/n, 13071 Ciudad Real, Spain. \\ \vspace{5mm}
\end{center}
\begin{Abstract}
In this short communication we answer negatively to a question posed by the authors in their work \cite[Section 1.4]{Brue2022} about the complex interpolation nature of the spaces $BV^{s}(\R^n)$, introduced by G.E. Comi and G. Stefani in \cite{ComiStefani2019}.
\end{Abstract}

\tableofcontents

\section{Introduction}
In \cite{ComiStefani2019} it was introduced the \textit{space of functions with bounded fractional variation} $BV^s(\R^n)$ as $$BV^s(\R^n):=\{f\in L^1(\R^n): [f]_{BV^s}<+\infty\},$$ where the fractional variation is defined as $$[f]_{BV^s}:=\operatorname{sup}\Bigg\{\int_{\R^n}f\operatorname{div}^s\varphi: \varphi\in C_c^\infty(\R^n;\R^n), \norm{\varphi}_\infty\leq 1\Bigg\}.$$ Those spaces are introduced as the fractional counterpart of the classical space of functions with bounded variation $BV(\R^n)$, which is defined analogously for the classical gradient $\nabla$. 
The works \cite{Brue2022, ComiStefani2019, ComiStefani2023} are devoted to the study of these fractional spaces. In particular, in \cite[Theorem 4.2]{Brue2022}, it is obtained that
for every $f\in \mathcal{H}^1(\R^n)\cap BV^s(\R^n)$, we have that \begin{equation}\label{eqn:ineq}[f]_{BV^t}\leq C(n,s)\norm{f}_{\mathcal{H}^1}^{\frac{s-t}{s}}[f]_{BV^s}^{t/s},\,0\leq t<s\leq 1,\end{equation} where $\mathcal{H}^1(\R^n)$ is the real Hardy space $$\mathcal{H}^1(\R^n):=\{f\in L^1(\R^n):\mathcal{R}f\in L^1(\R^n;\R^n)\},$$ with $$\mathcal{R}_jf(x):=\pi^{\frac{-(n+1)}{2}}\Gamma\left(\frac{n+1}{2}\right)\lim_{\varepsilon\to 0^+}\int_{B_\varepsilon^c}\frac{(y_j-x_j)f(x)}{|y-x|^{n+1}}\,dy,\,x\in\R^n,\,j=1,\ldots,n,$$ being the Riesz transform.

This type of inequalities are known as \textit{interpolation inequalities}, and usually they can be obtained by means of interpolation methods such as the Calder\'on's complex interpolation method. In fact, in \cite[Section 1.4]{Brue2022} it is posed if the space $BV^s(\R^n)$ is a complex interpolation space between $\mathcal{H}^1(\R^n)$ and $BV(\R^n)$, and hence if one can obtain the inequality \ref{eqn:ineq} by means of abstract interpolation. This question is very natural if one looks at the interpolation inequalities obtained in \cite{Brue2022} and many other embeddings supporting the question that the authors proved (see \cite[Section 1.4]{Brue2022}). Moreover, we already know that for $1<p<\infty$, the space of $p$-summable functions with $p$-summable fractional gradients (also in more general settings as weighted $L^p$ spaces or Musielak-Orlicz spaces) are a complex interpolation spaces, the Bessel potential spaces $H^{s,p}(\R^n)$ (see \cite{BellidoGarcia2025, BellidoCuetoGarcia2025, CamposGarciaSaez2026, Garcia2025}), so the Riesz fractional gradient is in fact a nonlocal operator deeply linked with the complex interpolation method. 
\section{Preliminaries}
We briefly summarize the tools that we will require for our result.
\subsection{Complex interpolation}
In this subsection we briefly briefly introduce Interpolation Theory, including the concepts and results we will use in the following. For a complete development of the main ideas, we refer to \cite{BerghLofstrom1976,GarciaSaez2024,Lunardi2018,Triebel1995}.

Let $(E_0,E_1)$ a couple of Banach spaces. We say that the couple is \textit{compatible} if there exists a Hausdorff topological vector space $\mathcal{E}$ such that $E_0,E_1\xhookrightarrow{}\mathcal{E}$. We say that a Banach space $E$ is intermediate with respect to the couple if $E_0\cap E_1\xhookrightarrow{}E\xhookrightarrow{}E_0+E_1$. Let $(F_0,F_1)$ another compatible couple and $T:E_0+E_1\to F_0+F_1$ bounded a linear. We say that $T$ is \textit{admissible} if $T:E_j\to F_j$, $j=0,1,$ continuously. Given $E$ an intermediate space with respect to the couple $(E_0,E_1)$, and $F$ an intermediate space with respect to $(F_0,F_1)$, we say that $E,F$ are \textit{interpolation spaces} with respect to the couples $(E_0,E_1)$ and $(F_0,F_1)$, respectively, if for every admissible operator $T:E_0+E_1\to F_0+F_1$, $T:E\to F$ continuously. The methods to construct such interpolation spaces for given couples are called \textit{interpolation functors.} We refer to \cite{BerghLofstrom1976,GarciaSaez2024,Lunardi2018,Triebel1995} for detailed expositions on the interpolation theory.

We say that an interpolation functor $\mathcal{F}$ is of \textit{exponent} $\theta\in (0,1)$ if $$\norm{T}_{\mathcal{F}\left((E_0,E_1)\right)\to \mathcal{F}\left((F_0,F_1)\right)}\leq C\norm{T}_{E_0\to F_0}^{1-\theta}\norm{T}_{E_1\to F_1}^\theta,$$ for some positive constant $C$. If we can choose $C=1$, we say that the functor is \textit{exact} of exponent $\theta$.

For our purposes, we will focus on the complex interpolation method. Given a compatible couple of Banach spaces $(E_0,E_1)$, we define the space $\mathfrak{F}(E_0,E_1)$, as the space of functions $f:S\to E_0+E_1$, where $S:=\{z\in \C: 0\leq \operatorname{Re}z\leq 1\}$, such that $f$ is holomorphic on the interior of $S$, continuous and bounded on $S$, and the functions $t\mapsto f(j+it)$, $j=0,1$, are continuous from $\R\to E_j$, and such that $\norm{f(j+it)}_{E_j}\to 0$ as $|t|\to \infty$. The space $\mathfrak{F}(E_0,E_1)$ is a vector space which is complete endowed with the norm $$\norm{f}_{\mathfrak{F}(\overline{E})}:=\operatorname{max}\{\operatorname{sup}_{t\in\R}\norm{f(it)}_{E_0},\operatorname{sup}_{t\in\R}\norm{f(1+it)}_{E_1}\},\,f\in \mathfrak{F}(\overline{E}).$$ From this space, we construct the \textit{complex method} as the functor $\mathcal{C}_\theta$, $\theta\in [0,1]$ which associates the space $[E_0,E_1]_\theta$ to the compatible couple of Banach spaces $(E_0,E_1)$. The space $[E_0,E_1]_\theta$ is defined as the space of $x\in E_0+E_1$ such that there exists $f\in \mathfrak{F}(E_0,E_1)$ with $f(\theta)=x$. The space is a Banach space endowed with the norm $$\norm{x}_{\theta}:=\operatorname{inf}\{\norm{f}_{\mathfrak{F}(\overline{E})}: f\in \mathfrak{F}(\overline{E}), f(\theta)=x\}.$$ 

Now we enumerate the method's properties we will require
\begin{Teor}[Properties of Complex interpolation spaces]\label{ComplexProps}
        Let $(E_0,E_1)$ a compatible couple of Banach spaces, and $\theta\in [0,1]$. Then, we have
        \begin{enumerate}
        \item[1.-] The spaces $[E_0,E_1]_\theta$ are exact interpolation spaces of exponent $\theta$
            \item[2.-] $E_0\cap E_1$ is dense in $[E_0,E_1]_\theta$.
            \item[3.-] There exists a positive constant $C>0$ such that for every $u\in E_0\cap E_1$, $$\norm{u}_{[E_0,E_1]_\theta}\leq C\norm{u}_{E_0}^{1-\theta}\norm{u}_{E_1}^\theta.$$
         \end{enumerate}
\end{Teor}
Detailed proofs of these facts can be found in \cite[Theorem 4.1.2, Theorem~4.2.1, Theorem~4.2.2, Theorem~4.5.1]{BerghLofstrom1976} and \cite[Theorem, IV.1.5, Proposition~IV.1.8, Theorem~IV.5.4, Theorem~IV.5.6]{GarciaSaez2024}.
\subsection{Functions of Bounded Fractional Variation}
Since we will be dealing with weighted fractional Sobolev spaces, we recall the basic definitions of the fractional operators introduced by Shieh and Spector in \cite{ShiehSpector2015,ShiehSpector2018}. Let $u\in C_c^\infty(\R^n)$, $v\in C_c^\infty(\R^n;\R^n)$ and $s\in (0,1)$. We define respectively the \textit{Riesz fractional gradient} $\nabla^s$ and the \textit{fractional divergence} as $$\nabla^su(x)=c_{n,s}\int_{\R^n}\frac{u(x)-u(y)}{|x-y|^{n+s}}\frac{x-y}{|x-y}\,dy,\,x\in \R^n,\,\operatorname{div}^sv(x)=c_{n,s}\int_{\R^n}\frac{v(x)-v(y)}{|x-y|^{n+s}}\cdot\frac{x-y}{|x-y}\,dy,\,x\in \R^n,$$ where $c_{n,s}$ is a normalizing constant. One of the main properties of these objects is that they can be seen as the classical gradient and divergence of the Riesz potential $$I_su:=I_{s}*u(x),\,I_s(x):=\frac{1}{\gamma_{n,s}}\frac{1}{|x|^{n-s}},\,0<s<n.$$  of the functions $u$ and $v$, i.e., $$\nabla^s u=\nabla(I_su)=I_s(\nabla u),\, \operatorname{div}^sv=\operatorname{div}(I_s v)=I_s(\operatorname{div}v).$$
See \cite{ShiehSpector2015, ShiehSpector2018} for the main properties of those operators.

We define \textit{space of functions with bounded fractional variation} $BV^s(\R^n)$ as $$BV^s(\R^n):=\{f\in L^1(\R^n): [f]_{BV^s}<+\infty\},$$ where the fractional variation is defined as $$[f]_{BV^s}:=\operatorname{sup}\Bigg\{\int_{\R^n}f\operatorname{div}^s\varphi: \varphi\in C_c^\infty(\R^n;\R^n), \norm{\varphi}_\infty\leq 1\Bigg\}.$$ See \cite{Brue2022, ComiStefani2019, ComiStefani2023} for its main properties. In particular we are interested in the fact that the compactly supported smooth functions $C_c^\infty(\R^n)$ are dense in $BV^s(\R^n)$ \cite[Theorem 3.8]{ComiStefani2019}.

\section{The spaces $[\mathcal{H}^1(\R^n), BV(\R^n)]_s$ and $BV^s(\R^n)$}
Motivated by the interpolation inequality \ref{eqn:ineq}, the authors in \cite{Brue2022} ask if the space $BV^s(\R^n)$ is related to the complex interpolation space $[\mathcal{H}^1(\R^n),BV(\R^n)]_s$, in particular if they are equal of at least if $BV^s(\R^n)\xhookrightarrow{}[\mathcal{H}^1(\R^n),BV(\R^n)]_s$, because that would imply by Theorem \ref{ComplexProps}(iii) the inequality \ref{eqn:ineq}. In this section we answer negatively to this fact by a simple but not that trivial observation.

To show that these two spaces are not equal, we will show that every function on $X_s:=[\mathcal{H}^1(\R^n), BV(\R^n)]_s$ integrates zero in the whole space, which is in contradiction with the fact that $C_c^\infty(\R^n)$ functions are dense in the second space. 

First of all we recall that every function $f\in \mathcal{H}^1(\R^n)$ integrates zero, i.e., \begin{equation}\label{eqn:h1}\int_{\R^n}f(x)\,dx=0,\,\forall f\in \mathcal{H}^1(\R^n),\end{equation} as it is proved in \cite[Corollary 6.7.7]{Grafakos2}. Although this property does not hold for $BV$-functions, it is carried out by the complex interpolation method for the whole space $X_s$.
\begin{Teor}
    Let $0<s<1$. Then, for every $f\in X_s$, we have that $$\int_{\R^n}f(x)\,dx=0.$$
\end{Teor}
\noindent{}\textbf{Proof:}
Let us define the linear operator $$l(f):=\int_{\R^n}f(x)\,dx,$$ for functions $f\in \mathcal{H}^1(\R^n)+BV(\R^n)$. This operator is well defined and it is continuous since $$|l(f)|\leq \norm{f}_1\leq \norm{f}_{\mathcal{H}^1+BV},$$ since both $\mathcal{H}^1(\R^n)$ and $BV(\R^n)$ are proper subspaces of $L^1(\R^n)$. Now, fixed $f\in X_s$, there must exist a function $F\in \mathfrak{F}\left(\mathcal{H}^1,BV\right)$ such that $F(s)=f$, with $$F(it)\in\mathcal{H}^1(\R^n),\,F(1+it)\in BV(\R^n),\,t\in\R.$$
Now, we define the operator $\Phi(z):=l(F(z))$, for $z$ in $\overline{S}$, which is holomorphic on the interior of $S$, continuous on $S$ and bounded on the boundary of the strip.  In fact, $$\Phi(it)=l(F(it))=\int_{\R^n}F(it)\,dx=0,\,t\in\R,$$ since $F(it)\in \mathcal{H}^1(\R^n)$, and $$|\Phi(1+it)|\leq \norm{F(1+it)}_1\leq \norm{F(1+it)}_{BV},$$ so there exists a positive constant $M$ such that $|\Phi(1+it)|\leq M$ for every $t\in \R$. Hence, by the Hadamard's three line theorem, we have that $\Phi(s)=0$, and hence $\int_{\R^n}f(x)\,dx=0$, as we wanted to prove. \qed 

Since by Theorem \ref{ComplexProps}(ii),  $\mathcal{H}^1\cap BV$ is dense on $X_s$, we can easily derive as well the last result in an easier fashion. 

Since by \cite[Theorem 3.8]{ComiStefani2019}, we have the density of $C_c^\infty(\R^n)$ on $BV^s(\R^n)$, we can choose a function $\varphi\in C_c^\infty(\R^n)$ such that $\varphi\geq 0$, $\varphi\not =0$, and hence $\int_{\R^n}\varphi>0$. By the previous result, the conclusion follows directly:
\begin{Coro}
    For every $s\in (0,1)$, we have that $BV^s(\R^n)\not\subset X_s$.
\end{Coro}

\section*{Acknowledgements}
This work was supported by Agencia Estatal de Investigación (Spain) through grant PID2023-151823NB-I00 and Junta de Comunidades de Castilla-La Mancha (Spain) through grant SBPLY/23/180225/000023. G.G.-S. is supported by a Doctoral Fellowship by \textit{Universidad de Castilla-La Mancha} \text{2024-UNIVERS-12844-404}.

The author would like to thank to his advisor José Carlos Bellido for the useful suggestions
on a preliminary version of the manuscript, and O. Domínguez for several conversations on the subject of this paper.


\section*{Conflicts of interest}

The authors declare that there are no conflicts of interest regarding the publication of this paper.

\addcontentsline{toc}{section}{References}
\bibliographystyle{plain}

\end{document}